\documentclass[12pt,a4paper]{article}
\usepackage[utf8]{inputenc}
\usepackage[scale=0.7]{geometry}
\usepackage{amsmath,amsthm,amssymb,enumerate}
\usepackage{graphicx}
\usepackage{tabularx,booktabs}
\usepackage{mathtools}
\newcolumntype{Y}{>{\centering\arraybackslash}X}

\DeclarePairedDelimiter\floor{\lfloor}{\rfloor}

\newtheorem{theorem}{Theorem}[section]
\newtheorem{lemma}[theorem]{Lemma}

\newtheorem{corollary}[theorem]{Corollary}

\long\def\delete#1{}

\usepackage{authblk}

\usepackage{latexsym,amsmath,amsfonts,amssymb,amsthm}

\skip\footins=8mm plus 1mm

\usepackage{longtable, booktabs}
\usepackage{supertabular}
\usepackage{rotating}
\usepackage{multicol}
\usepackage{multirow}

\usepackage{hyperref}
\usepackage[labelsep=period]{caption}

\usepackage{color}

\definecolor{VeryLightBlue}{rgb}{0.9,0.9,1}
\definecolor{LightBlue}{rgb}{0.8,0.8,1}
\definecolor{MidBlue}{rgb}{0.5,0.5,1}
\definecolor{DarkBlue}{rgb}{0,0,0.6}
\definecolor{Blue}{rgb}{0,0,1}

\definecolor{Gold}{rgb}{1,0.843,0}

\definecolor{LightGreen}{rgb}{0.88,1,0.88}
\definecolor{MidGreen}{rgb}{0.6,1,0.6}
\definecolor{DarkGreen}{rgb}{0,0.6,0}

\definecolor{VeryLightYellow}{rgb}{1,1,0.9}
\definecolor{LightYellow}{rgb}{1,1,0.6}
\definecolor{MidYellow}{rgb}{1,1,0.5}
\definecolor{DarkYellow}{rgb}{1,1,0.2}

\definecolor{DarkPurple}{rgb}{.6,0,1}

\definecolor{Red}{rgb}{1,0,0}
\definecolor{VeryLightRed}{rgb}{1,0.9,0.9}
\definecolor{LightRed}{rgb}{1,0.8,0.8}
\definecolor{MidRed}{rgb}{1,0.55,0.55}

\newcommand{\be}{\begin{equation}}
\newcommand{\ee}{\end{equation}}
\newcommand{\bea}{\begin{eqnarray}}
\newcommand{\eea}{\end{eqnarray}}
\newcommand{\bean}{\begin{eqnarray*}}
\newcommand{\eean}{\end{eqnarray*}}

\def\PG{\mathrm{PG}}

\def\Aut{{\rm Aut}}
\def\PSU{{\rm PSU}}
\def\PGammaU{{\rm P\Gamma U}}

\def\qed{\hfill$\Box$\vspace{12pt}}

\usepackage[normalem]{ulem}

\title{The vertex-isoperimetric number of the incidence and non-incidence graphs of unitals}

\author[1]{Alice M. W. Hui\thanks{E-mail: \texttt{alicemwhui@uic.edu.hk}, \texttt{huimanwa@gmail.com}}}
\author[2]{Muhammad Adib Surani\thanks{E-mail: \texttt{m.surani@student.unimelb.edu.au}}}
\author[2]{Sanming Zhou\thanks{E-mail: \texttt{sanming@unimelb.edu.au}}}
\affil[1]{\small Statistics Program, Beijing Normal University-Hong Kong Baptist University United International College, China}
\affil[2]{\small School of Mathematics and Statistics, The University of Melbourne, Parkville, VIC 3010, Australia}

\begin{document}
\openup 0.8\jot 
\maketitle

\abstract{We derive upper and lower bounds for the vertex-isoperimetric number of the incidence graphs of unitals and determine its order of magnitude. In the case when a unital contains sufficiently large arcs, these bounds agree and give rise to the precise value of this parameter. In particular, we obtain the exact value of the vertex-isoperimetric number of the incidence graphs of classical unitals and a certain subfamily of BM-unitals. In the case when the maximum size of arcs in the unital is relatively small, we obtain an upper bound for this parameter in terms of the vertex-isoperimetric number of the incidence graph. We also determine the exact value of the vertex-isoperimetric number of the non-incidence graph of any unital.

\medskip


}

\delete
{
\abstract{Let $U$ be a unital and $G_U$ its incidence graph. Let $i(G_U)$ be the vertex-isoperimetric number of $G_U$. We obtain upper and lower bounds for $i(G_U)$ and determine its order of magnitude. In the case when $U$ contains sufficiently large arcs, these bounds agree and give rise to the precise value of $i(G_U)$. In particular, we obtain the exact value of $i(G_U)$ when $U$ is any classical unital or in a certain subfamily of BM-unitals. In the case when the maximum size of arcs in $U$ is relatively small, we obtain an upper bound for this parameter in terms of $i(G_U)$. We also determine the exact value of the vertex-isoperimetric number of the non-incidence graph of any unital.}
}

\section{Introduction}
\label{sec:int}

A \emph{unital} is a $2$-$(n^3+1, n+1, 1)$ design for some integer $n \ge 2$. In this paper we derive upper and lower bounds for the vertex-isoperimetric number of the incidence graphs of unitals and determine its order of magnitude. In the case when a unital contains sufficiently large arcs, we obtain the precise value of the vertex-isoperimetric number of its incidence graph. In particular, this happens for all classical unitals and a certain subfamily of BM-unitals, and from the former we obtain a new infinite family of edge-transitive graphs whose vertex-isoperimetric number can be computed exactly.
In the case when the maximum size of arcs in the unital is relatively small, our upper bound can be interpreted as an upper bound for this parameter in terms of the vertex-isoperimetric number of the incidence graph. We also determine the precise value of the vertex-isoperimetric number of the non-incidence graph of any unital. 

Let $G$ be a graph with vertex set $V(G)$ and $S$ a subset of $V(G)$. The \emph{neighbourhood} of $S$ in $G$, denoted by $N_G(S)$ or simply $N(S)$, is the set of vertices in $V(G) \setminus S$ that are adjacent to at least one vertex in $S$. The \emph{vertex-isoperimetric number} of $G$ is defined as
\[
i(G) = \min\left\{\frac{|N(S)|}{|S|} : \emptyset \neq S \subseteq V(G), |S| \le \frac{|V(G)|}{2}\right\}.
\]
The \emph{edge-isoperimetric number}  (also known as the Cheeger constant) of $G$ is defined in the same way but with $N(S)$ replaced by the set of edges between $S$ and $V(G) \setminus S$. 
These two parameters are of fundamental importance for measuring connectivity and expansion of graphs. For example, they are used to define expander graphs, which play a prominent role in theoretical computer science, information theory, cryptography, combinatorics, pure mathematics, etc. The isoperimetric problems for graphs have been studied extensively. See \cite{HLW} for background information and a large number of results (often involving eigenvalues of graphs) in the context of expander graphs. See also \cite{harper2004global} for the history of the vertex-isoperimetric problem and related results. It is NP-hard to determine the vertex-isoperimetric number of a general graph, and there are few known families of graphs whose vertex-isoperimetric numbers have been computed explicitly. 

Given positive integers $v > k > \lambda$, a \emph{$2$-$(v,k,\lambda)$ design} \cite{colbourn2010crc} is a pair $D = (P, \mathcal{B})$, where $P$ is a set of $v$ elements each called a \emph{point}, and $\mathcal{B}$ is a set of $k$-element subsets of $P$ each called a \emph{block}, such that any 2-element subset of $P$ is contained in exactly $\lambda$ blocks. The number of blocks containing a given point is independent of the choice of the point. This number is denoted by $r$ and the number of blocks is denoted by $b$. We call $(v, b, r, k, \lambda)$ the \emph{parameters} of $D$. These parameters always satisfy $vr=bk$ and $r(k-1) = \lambda(v-1)$ \cite{colbourn2010crc}. The \emph{complement} of $D$ is defined \cite{colbourn2010crc} as $\overline{D} = \left(P,\overline{\mathcal{B}}\right)$, where $\overline{\mathcal{B}} = \{P\setminus B : B \in \mathcal{B}\}$; $\overline{D}$ is a $2$-design with parameters $(v,b,b-r,v-k,b-2r+\lambda)$.

The \emph{incidence graph} of a $2$-design $D = (P, \mathcal{B})$, denoted by $G_{D}$, is the bipartite graph with vertex set $P \cup \mathcal{B}$ and bipartition $\{P, \mathcal{B}\}$, such that a point vertex $p \in P$ is adjacent to a block vertex $B \in \mathcal{B}$ if and only if $p$ lies in $B$. Similarly, the \emph{non-incidence graph} of $D$ is the bipartite graph with the same vertex set and bipartition such that $p \in P$ is adjacent to $B \in \mathcal{B}$ if and only if $p$ does not lie in $B$. The non-incidence graph of $D$ is the same as the incidence graph of $\overline{D}$ and hence is denoted by $G_{\overline{D}}$.

A unital of \emph{order} $n \ge 2$ is a 2-$(n^3+1, n+1, 1)$ design \cite{colbourn2010crc}; its parameters are 
\be
\label{eq:uni}
(v,b,r,k,\lambda) = (n^3+1, n^4-n^3+n^2, n^2, n+1, 1).
\ee
The blocks of a unital are often called \emph{lines}, and we use these two terms interchangeably.
Unitals are interesting because some of them are embeddable in a projective plane as subdesigns \cite{BE3}. There are two well known families of unitals embedded in the Desarguesian projective plane $\PG(2,q^2)$ over the finite field $\mathbb F_{q^2}$ of order $q^2$, where $q > 2$ is a prime power. One is the family of \emph{classical unitals} $H(q)$, formed by the absolute points and non-absolute lines of a unitary polarity of $\PG(2,q^2)$. The points of $H(q)$ are the points of a Hermitian curve, which is projectively equivalent to the curve
$$
\left\{(x_0, x_1,x_2) : x_0^{q+1}+ x_1^{q+1} + x_2^{q+1} = 0, x_0, x_1,x_2\in \mathbb F_{q^2}\right\}.
$$
It is well known \cite{O'Nan, Taylor} that $\Aut(H(q)) = \PGammaU(3, q)$. Thus, for every $G$ with $\PSU(3,q) \le G \le \PGammaU(3,q)$, $H(q)$ is $(G, 2)$-point-transitive and hence $G$-flag-transitive, where a {\em flag} is an incident point-line pair. 

Another family of unitals embeddable in $\PG(2,q^2)$ consists of \emph{Buekenhout-Metz unitals} (thereafter BM-unitals) arising from two constructions \cite{Bkt,Me} of Buekenhout. These unitals are called \emph{parabolic} or \emph{hyperbolic}, depending on the number of unital points on the line at infinity of $\PG(2,q^2)$.
When $q$ is odd, the point set of a parabolic BM-unital is of the form
\begin{equation}\label{BMunital}
	\left\{(x, \alpha x^2+ \beta x^{q+1}+r,1) : x\in \mathbb F_q^2,r \in \mathbb F_q\right\}\cup \{(0,1,0)\},
\end{equation}
where $\alpha, \beta \in \mathbb F_{q^2}$, and $(\beta^q-\beta )^2 +4\alpha^{q+1}$ is a non-square in $\mathbb F_q$ \cite{BE2}. When $q$ is even, the point set of a parabolic BM-unital is of the form \eqref{BMunital}, where $\alpha, \beta \in \mathbb F_{q^2}, \beta \notin \mathbb F_q$, and $\alpha^{q+1}/(\beta^q+\beta)^2$ has trace 0 over $\mathbb F_2$ \cite{E}. In both cases we denote the unital defined in \eqref{BMunital} by $U(\alpha, \beta,q)$.  


Given an integer $k \ge 3$, a \emph{$k$-arc}, or simply an \emph{arc} if its size $k$ is not a concern, of a unital $U$ is a set of $k$ points of $U$ of which no three are collinear. A $k$-arc is \emph{complete} if it is not contained in any $(k+1)$-arc. Complete $k$-arcs have been studied for several classes of unitals \cite{AK,BE,fisher1986complete,HS}. Their sizes range from $n-1$ to $n^2+1$, where $n$ is the order of the unital. Define
$$
m(U) = \max\{k: k \ge 3, \text{ $U$ contains a $k$-arc}\}
$$ 
to be the maximum size of an arc in $U$. 

The purpose of this paper is to study the vertex-isoperimetric number of the incidence and non-incidence graphs of unitals. 
Our first main result, stated below, asserts that the vertex-isoperimetric number of the incidence graph of any unital of order $n$ is $\frac{2n}{n^2 + 1}$ approximately. Set
\be
\label{eq:x*n}
c(n) = n^2 - \frac{\sqrt{8n^2+9}-3}{2}.
\ee

\begin{theorem}
\label{thm:inc}
Let $U$ be a unital of order $n \ge 2$. Then
\be
\label{eq:lub}
\frac{2 \left(n^3+1-\floor*{c(n)}\right)}{n^2(n^2+1)} \le i(G_U) \le \frac{2 \left(n^3+1-\min\{m(U), \floor*{c(n)}\}\right)}{n^2(n^2+1)}.
\ee
\end{theorem}

The upper bound here says that $i(G_U)$ is bounded from above by $m(U)$. In particular, if a unital $U$ admits sufficiently large arcs, namely $m(U) \ge \floor*{c(n)}$, then $i(G_U)$ is given by the following formula.  

\begin{corollary}
\label{cor:sharp}
Let $U$ be a unital of order $n \ge 2$. If $m(U) \ge \floor*{c(n)}$, then
	\[
	i(G_U) = \frac{2(n^3+1-\floor*{c(n)})}{n^2(n^2+1)}.
	\]
\end{corollary}

The case $m(U) \ge \floor*{c(n)}$ occurs for all classical unitals and a subfamily of BM-unitals. In fact, it was proved in \cite{fisher1986complete} that for any prime power $q$ the classical unital $H(q)$ admits a complete $(q^2-q+1)$-arc. (In fact, $H(q)$ is the disjoint union of $q+1$ complete arcs of size $q^2-q+1$.) Hence $m(H(q)) \ge q^2-q+1 \ge \floor*{c(q)}$. In \cite{BE,HS} it was proved that for an odd prime power $q$, if $\alpha\in \mathbb F_{q^2}$ and $\beta \in \mathbb F_q$ are such that $(\beta^q-\beta )^2 +4\alpha^{q+1}$ is a non-square in $\mathbb F_q$, then the BM-unital $U(\alpha, \beta,q)$ admits a complete $(q^2+1)$-arc and so $m(U(\alpha, \beta,q)) \ge \floor*{c(q)}$. Combining these known results and Corollary \ref{cor:sharp}, we obtain the following corollaries. 

\begin{corollary}
\label{cor:classical}
	Let $H(q)$ be the classical unital of order $q > 2$, where $q$ is a prime power. Then
	\[
	i(G_{H(q)}) = \frac{2(q^3+1-\floor*{c(q)})}{q^2(q^2+1)}.
	\]
\end{corollary}

\begin{corollary}
\label{cor:BM}
Let $q$ be an odd prime power, and let $\alpha\in \mathbb F_{q^2}$ and $\beta \in \mathbb F_q$ be such that $(\beta^q-\beta )^2 +4\alpha^{q+1}$ is a non-square in $\mathbb F_q$. Then
	\[
	i(G_{U(\alpha, \beta,q)}) = \frac{2(q^3+1-\floor*{c(q)})}{q^2(q^2+1)}.
	\]
\end{corollary}

Since $H(q)$ is flag-transitive, its incidence graph $G_{H(q)}$ is edge-transitive. So Corollary \ref{cor:classical} provides an infinite family of edge-transitive graphs whose vertex-isoperimetric numbers are known explicitly. 
 
Another way to interpret the upper bound in Theorem \ref{thm:inc} is as follows.  

\begin{corollary}
\label{cor:inc}
Let $U$ be a unital of order $n \ge 2$. Then either 
$$
m(U) > \floor*{c(n)}
$$ 
or 
$$
m(U) \le n^3+1 - \frac{n^2(n^2+1)}{2} i(G_U). 
$$
\end{corollary}

In other words, either $m(U)$ is large or it can be bounded from above by $i(G_U)$. 

Our second main result is the following formula for the vertex-isoperimetric number of the non-incidence graph of any unital.
 
\begin{theorem}
\label{thm:noninc}
Let $U$ be a unital of order $n \ge 2$. Then
\be
\label{eq:noninc}
i(G_{\overline{U}}) = \begin{cases}
        \frac{4}{5}, & \text{ if } n=2\\  
        \frac{2(n^3+1)}{n^2(n^2+1)}, & \text{ if }  n \ge 3.
    \end{cases}
\ee
\end{theorem}

Theorems \ref{thm:inc} and \ref{thm:noninc} will be proved in sections \ref{sec:inc_unital} and \ref{sec:noninc_unital}, respectively. In order to prove the lower bound in \eqref{eq:lub}, we will first derive lower bounds on the size of $N(S)$ in $G_D$ for a general $2$-design $D$ and any $S$ containing only point vertices or block vertices. These general lower bounds are of interest for their own sake as they may be used in studying the vertex-isoperimetric problem for other families of $2$-designs. 

We conclude this section by mentioning a few results relevant to this paper. Denote by $\Gamma_{n,q}$ the incidence graph of $\PG(n,q)$, where $n \ge 2$ and $q$ is a prime power. These graphs form a major subfamily of the family of $2$-arc transitive Cayley graphs of dihedral groups \cite{DMM08}. Lanphier et al. \cite{lanphier2006expansion} studied the edge-isoperimetric number of $\Gamma_{2,q}$, and Harper and Hergert \cite{harper1994isoperimetric} and Ure \cite{ure1996study} studied the problem of minimizing $|N(X)|$ for a subset of points $X$ in $\PG(2,q)$ with a given size. Determining the precise value of $i(\Gamma_{n,q})$ turns out to be a difficult problem, even when $n=2$ and $q$ is small. With Elvey-Price the last two authors of the present paper determined in \cite{ESZ} the order of magnitude of $1-i(\Gamma_{n,q})$ and found the precise value of $i(\Gamma_{2,q})$ for $q \le 16$. In \cite{HJ}, H\o holdt and Janwa studied eigenvalues and expansion of bipartite graphs. In particular, they obtained lower bounds on the edge-isoperimetric number of $G_D$ in terms of its eigenvalues for any $2$-design $D$.

\section{Lower bounds for $2$-designs}
\label{sec:lb}

We first prove the following general lower bounds.

\begin{theorem}
\label{thm:keytheorem}
Let $D = (P, \mathcal{B})$ be a $2$-$(v, k,\lambda)$ design of $b$ blocks, with $r$ blocks through each point, and let $G$ be the incidence graph $G_D$ of $D$. Let $m$ be any positive integer. Then for any $X \subseteq P$ and $Y \subseteq \mathcal{B}$,  
\be 
\label{eq:main1}
    |N(X)| \ge \frac{2rm-\lambda (|X|-1)}{m(m+1)}|X|; 
\ee
\be
\label{eq:main2}
    |N(X)| \ge \frac{r^2}{r + \lambda(|X|-1)}|X|; 
\ee
\be
\label{eq:main3}
    |N(Y)| \ge \frac{rk}{r^2-\lambda(b-|Y|)}|Y|; 
\ee
\be
\label{eq:main4}
    |N(X) \setminus Y| \ge \frac{4\lambda}{k^2}|X||P \setminus (X \cup N(Y))|. 
\ee
Moreover, equality in \eqref{eq:main1} holds if every block vertex in $N(X)$ contains exactly $m$ or $m+1$ points in $X$.
\end{theorem}
\begin{proof}
Denote by $F$ the set of edges of $G$ with one end-vertex in $X$ and the other end-vertex in $N(X)$. By double counting, we have
	\be
\label{eq:rx}
		r|X| = |F| = \smashoperator[r]{\sum_{B \in N(X)}} |B \cap X|.
	\ee
Denote by $Q$ the set of paths of length two in $G$ with both end-vertices in $X$. By double counting again, we obtain
	\be
\label{eq:xx}
		\lambda |X|(|X|-1) = |Q| = \smashoperator[r]{\sum_{B \in N(X)}} |B \cap X|(|B \cap X| - 1).
	\ee
Since $n^2 \ge n$ for any integer $n$, by \eqref{eq:rx} and \eqref{eq:xx} we obtain
	\begin{align*}
		0 &\le \smashoperator[r]{\sum_{B \in N(X)}}\left\{(|B \cap X| - m)^2 - (|B \cap X| - m)\right\}\\
		&= \smashoperator[r]{\sum_{B \in N(X)}} \left\{|B \cap X|(|B \cap X| - 1) - 2m|B \cap X| + m(m+1)\right\}\\
		&= \lambda |X|(|X| - 1) - 2rm|X| + m(m+1)|N(X)|.
	\end{align*}
Rearranging this inequality yields \eqref{eq:main1}. In addition, if $|B \cap X| \in \{m,m+1\}$ for each $B \in N(X)$, then the inequality above is sharp and so equality in \eqref{eq:main1} is achieved. 

Take $m=\left\lfloor\frac{\lambda(|X|-1)}{r}\right\rfloor + 1$. Then  
	\[
		\left(\frac{\lambda(|X|-1)}{r}+1-m\right)\left(m-\frac{\lambda(|X|-1)}{r}\right) \ge 0.
	\]
That is,
	\[
		\left(\frac{\lambda(|X|-1)}{r}+1\right)\left(2m-\frac{\lambda(|X|-1)}{r}\right) \ge m(m+1).
	\]
	Rearranging this inequality yields
	\[
		\frac{2rm-\lambda(|X|-1)}{m(m+1)} \ge \frac{r^2}{r+\lambda(|X|-1)}.
	\]
	This together with \eqref{eq:main1} implies \eqref{eq:main2}.
	
We now prove \eqref{eq:main3}. If $N(Y) = P$, then
	\[
	    |N(Y)| = v \ge \frac{rk|Y|}{r^2-\lambda(b-|Y|)}
    \]
and \eqref{eq:main3} holds. Assume that $N(Y) \neq P$. Then $Y \ne \cal B$ and so $b > |Y|$. Note that $N(P \setminus N(Y)) \subseteq \mathcal{B} \setminus Y$. This together with \eqref{eq:main2} for $X = P \setminus N(Y)$ yields
\be
\label{eq:bY}
		b - |Y| \ge |N(P \setminus N(Y))| \ge \frac{r^2}{r+\lambda (v - |N(Y)| - 1)} (v - |N(Y)|) = \frac{r^2}{\frac{r - \lambda}{v - |N(Y)|} + \lambda}.
\ee
Since $vr=bk$ and $r(k-1) = \lambda(v-1)$, we have $\frac{r^2}{\lambda b} = \frac{k(v-1)}{v(k-1)} > 1$. So $\frac{r^2}{b-|Y|} - \lambda > 0$. Thus, by \eqref{eq:bY},  
	\begin{align*}
		|N(Y)| &\ge v - \frac{r-\lambda}{\frac{r^2}{b-|Y|} - \lambda}\\
		&= \frac{vr^2-(b-|Y|)(\lambda(v-1)+r)}{r^2-\lambda(b-|Y|)}\\
		&= \frac{rk}{r^2-\lambda(b-|Y|)}|Y|
	\end{align*}
and \eqref{eq:main3} is established.

Finally, let $A$ be the set of paths of length two in $G$ with one end-vertex in $X$ and the other end-vertex in $P \setminus (X \cup N(Y))$. Since each pair of points in $X \times (P \setminus (X \cup N(Y)))$ are contained in exactly $\lambda$ blocks, we have
	\begin{equation}
		|A| = \lambda|X||P \setminus (X \cup N(Y))|. \label{eq:firstAequation}
	\end{equation}
	
	Consider a block $B$ that is the middle vertex of a path in $A$. Since $B$ is adjacent to a point in $X$, we have $B \in N(X)$. Similarly, $B$ is adjacent to a point outside $N(Y)$, and so $B \notin Y$. Hence $B \in N(X) \setminus Y$. Denote by $a_B$ the number of points in $X$ adjacent to $B$ in $G$, and $b_B$ the number of points in $P \setminus (X \cup N(Y))$ adjacent to $B$ in $G$. Then $a_B + b_B \le k$ and so $a_B b_B \le \left(\frac{k}{2}\right)^2$. It follows that each such block $B$ can contribute at most $\left(\frac{k}{2}\right)^2$ paths to $A$. Since there are at most $|N(X) \setminus Y|$ such blocks, it follows that
	\[
		|A| \le \frac{k^2}{4} |N(X) \setminus Y|.
	\]
	This together with \eqref{eq:firstAequation} implies \eqref{eq:main4}.
\qed
\end{proof}

We now introduce notations that will be used in subsequent discussions. Let $D = (P, \mathcal{B})$ be a $2$-design and $G$ its incidence graph. If $X \subseteq P$ and $Y \subseteq \mathcal{B}$, set
\begin{equation}
\label{eq:shorthand}
x=|X|,\quad y=|Y|,\quad x'=|N(Y) \setminus X|, \quad y'=|N(X) \setminus Y|.
\end{equation}
Thus, if $S$ is a subset of the vertex set of $G$, by setting $X = S \cap P$ and $Y = S \cap \mathcal{B}$, we have
$$
|S|=x+y,\;\, |N(S)|=x'+y'.
$$

\begin{corollary}
\label{lem:keycorollary}
Let $D = (P, \mathcal{B})$ be a $2$-$(v, k,\lambda)$ design of $b$ blocks, with $r$ blocks through each point, and let $G$ be the incidence graph $G_D$ of $D$. Let $m$ be any positive integer. Then for any $X \subseteq P$ and $Y \subseteq \mathcal{B}$, and $x$, $y$, $x'$ and $y'$ as defined in \eqref{eq:shorthand}, 
$$
    y+y' \ge \frac{x(2rm-\lambda(x-1))}{m(m+1)}; 
$$
$$
    y+y' \ge \frac{r^2 x}{r+\lambda(x-1)};
$$
$$
    x+x' \ge \frac{rky}{r^2-\lambda(b-y)}; 
$$
$$
    y' \ge \frac{4\lambda}{k^2}(v-x-x').
$$
\end{corollary}
\begin{proof}
Since $y+y' \ge |N(X)|$ and $x+x' \ge |N(Y)|$, each inequality follows from the corresponding inequality in Theorem \ref{thm:keytheorem} immediately.
\end{proof}


\section{Proof of Theorem \ref{thm:inc}}
\label{sec:inc_unital}

Throughout this section we assume that $U=(P,\mathcal{B})$ is a unital of order $n \ge 2$ and $G$ its incidence graph $G_U$. Denote the vertex set of $G$ by $V = P \cup \mathcal{B}$. Recall that the parameters of $U$ are given by \eqref{eq:uni}. 
Since $r = n^2$ and $\lambda = 1$, applying \eqref{eq:main1} to $U$ and $m = 1$ yields $|N(X)| \ge n^2 x - \frac{x (x - 1)}{2}$ for $X \subseteq P$ with $|X| = x$. Motivated by this, we define $g: [0,n^2+1] \rightarrow \left[0,\frac{(n^2+1)(n^2+2)}{2}\right]$ by
\[
g(z) = (n^2 + 1)z - \frac{z(z-1)}{2}. 
\]
Then $|N(X)| \ge g(x)-x$ for any subset $X$ of $P$ with size $x$. Moreover, by Theorem \ref{thm:keytheorem}, $|N(X)| = g(x)-x$ if $X$ is an $x$-arc of $U$ since every line in $N(X)$ contains exactly one or two points in $X$. It can be verified that the function $g$ is increasing, concave, and bijective in the domain $[0,n^2+1]$. The function $c(n)$ introduced in \eqref{eq:x*n} can be expressed as 
\[
c(n) = g^{-1}\left(\frac{n^2 (n^2 + 1)}{2}\right).
\]

\subsection{Upper bound}

We now prove the upper bound in Theorem \ref{thm:inc} by an explicit construction. Denote 
$$
x = \min\{m(U), \floor*{c(n)}\}.
$$ 
Take a subset $X$ of some $m(U)$-arc of $U$ with size $x$. Then $X$ is an $x$-arc of $U$. Since $g$ is increasing in $[0,n^2+1]$, and $x \le c(n)$, from the discussion above we obtain 
	\[
	|N(X)| = g(x)-x \le g\left(c(n)\right)-x = \frac{n^2 (n^2 + 1)}{2}-x.
	\]
So we can take a subset $Y$ of $\mathcal{B}$ such that $N(X) \subseteq Y$ and $|Y| = \frac{n^2 (n^2 + 1)}{2} - x$. It then follows that
	\[
	N(X \cup Y) = N(Y) \subseteq P \setminus X.
	\]
Hence
	\[
	\frac{|N(X \cup Y)|}{|X \cup Y|} \le \frac{n^3+1-x}{\frac{n^2 (n^2 + 1)}{2}}.
	\]
From this the upper bound in \eqref{eq:lub} follows immediately. 

\subsection{Lower bound}

The lower bound in Theorem \ref{thm:inc} will be proved by bounding $|N(S)|$ from below for an arbitrary $S \subseteq V$ with $1 \le |S| \le \frac{|V|}{2}$. Let $X = S \cap P$ and $Y = S \cap \mathcal{B}$, and let $x$, $y$, $x'$, and $y'$ be defined as in \eqref{eq:shorthand}. Then $|S| = x+y$ and $|N(S)| = x'+y'$. Since $|S| \le \frac{|V|}{2}$, we have $x+y \le \frac{n^2 (n^2+1)}{2}$. Setting $m = 1$ in Corollary \ref{lem:keycorollary}, we obtain
$$
y + y' \ge \frac{x(2n^2+1-x)}{2};
$$
$$
y + y' \ge \frac{n^4 x}{n^2 - 1 + x};
$$
$$
x + x' \ge \frac{n^2(n+1)y}{n^2(n-1)+y};  
$$
$$
y' \ge \frac{4x(n^3+1-x-x')}{(n+1)^2}.
$$
  
Define
\[
f(x,y) = \frac{4x}{(n+1)^2}(n^3+1) + \left(1 - \frac{4x}{(n+1)^2}\right)\frac{n^2(n+1)y}{n^2(n-1)+y} - x.
\]
It is readily seen that $f(x,y)$ is linear in $x$ for a fixed $y$, and monotonic in $y$ for a fixed $x$.  

\begin{lemma}
\label{lem:two}
If either 
\be
\label{eq:cond1}
x \le \frac{(n+1)^2}{4} \quad\text{and}\quad \frac{f(x,y)}{x+y} \ge \frac{2(n^3+1-\floor*{c(n)})}{n^4+n^2},
\ee
or
\be
\label{eq:cond2}
x > \frac{(n+1)^2}{4} \quad\text{and}\quad x - \left(1-\frac{(n+1)^2}{4x}\right)y' < \floor*{c(n)}+1, 
\ee
then
\be
\label{eq:xyxy}
\frac{x'+y'}{x+y} \ge \frac{2(n^3+1-\floor*{c(n)})}{n^2(n^2 + 1)}.
\ee
\end{lemma}

\begin{proof}
If \eqref{eq:cond1} holds, then
	\begin{align*}
	x'+y' &= \left(\frac{4x}{(n+1)^2}x' + y'\right) + \left(1 - \frac{4x}{(n+1)^2}\right) x'\\
	&\ge \frac{4x}{(n+1)^2}\left(n^3+1-x\right) + \left(1 - \frac{4x}{(n+1)^2}\right)\left(\frac{n^2(n+1)y}{n^2(n-1)+y}-x\right)\\
	&= f(x,y),
	\end{align*}
which together with \eqref{eq:cond1} implies \eqref{eq:xyxy}. 
	 	
On the other hand, if \eqref{eq:cond2} holds, then
	\begin{align*}
	x'+y' &= \left(x' + \frac{(n+1)^2}{4x}y'\right) + \left(1 - \frac{(n+1)^2}{4x}\right) y'\\
	&\ge n^3+1-x + \left(1 - \frac{(n+1)^2}{4x}\right)y'\\
	&> n^3-\floor{c(n)}.
	\end{align*}
Thus, $x'+y' \ge n^3+1-\floor*{c(n)}$, which together with $x+y \le \frac{n^2 (n^2+1)}{2}$ yields \eqref{eq:xyxy}.
\qed
\end{proof}

Now we are ready to prove the lower bound in Theorem \ref{thm:inc}. Let us begin with the case when $n = 2$. In this case we have $x+y \le \frac{n^2 (n^2+1)}{2} = 10$. Moreover, condition \eqref{eq:cond1} is reduced to
$$
	x \le 2 \quad\text{and}\quad \frac{f(x,y)}{x+y} \ge \frac{7}{10}, 
$$
and condition \eqref{eq:cond2} is reduced to
$$
	x \ge 3 \quad\text{and}\quad x - \left(1-\frac{9}{4x}\right)y' < 3. 
$$
By Lemma \ref{lem:two} it suffices to verify that one of these conditions holds. In fact, if $x=0$, then $\frac{f(x,y)}{x+y} = \frac{12}{4+y} > \frac{7}{10}$. If $x=1$, then $y \le 9$, so that $\frac{f(x,y)}{x+y} = \frac{29y+36}{3y^2+15y+12} > \frac{7}{10}$. If $x=2$, then $x'+y' \ge \frac{4x}{9}x'+y' \ge \frac{4x}{9}(9-x) > 6$, so that $\frac{x'+y'}{x+y} \ge \frac{7}{10}$. Finally, if $x \ge 3$, then
		\begin{align*}
		x - \left(1-\frac{9}{4x}\right)y'
		&= x - \left(1-\frac{9}{4x}\right)\left(x+(y+y')-(x+y)\right)\\
		&\le x - \left(1-\frac{9}{4x}\right)\left(x+\frac{16x}{3+x}-10\right)\\
		&= \frac{84}{3+x} - \frac{45}{2x} - \frac{15}{4}\\
		&< 3. \qedhere
		\end{align*}
Therefore, the lower bound in \eqref{eq:lub} holds when $n=2$. 

In the rest of this section we assume that $n \ge 3$. Then $8n^2+9 \le (3n)^2$ and hence
$$
c(n) \ge n^2-\frac{3}{2}(n-1) \ge \frac{n(n+1)}{2} \ge \frac{(n+1)^2}{4}.
$$
We consider the following five cases one by one. In each case we show that either \eqref{eq:cond1} or \eqref{eq:cond2} is satisfied and so the lower bound in \eqref{eq:lub} follows immediately from Lemma \ref{lem:two} and the arbitrariness of $S \subseteq V$ with $1 \le |S| \le \frac{|V|}{2}$. 

\smallskip
\textsf{Case 1:} $x \le \frac{(n+1)^2}{4}$ and $y \le \frac{n^4-n^3+n^2}{2}$. 

Since $f(x,y)$ is linear in $x$, we have
	\begin{align*}
	\frac{f(x,y)}{x+y} &\ge \min\left\{ \frac{f(0,y)}{y}, \frac{f\left(\frac{(n+1)^2}{4},y\right)}{\frac{(n+1)^2}{4}+y} \right\}\\
	&= \min\left\{ \frac{n^2(n+1)}{n^2(n-1)+y}, \frac{n^3+1-\frac{(n+1)^2}{4}}{\frac{(n+1)^2}{4} + y} \right\}\\
	&\ge \min\left\{ \frac{2(n+1)}{n^2+n-1}, \frac{(n+1)(4n^2-5n+3)}{(n+1)^2+2(n^4-n^3+n^2)} \right\}\\
	&\ge \frac{2}{n}\\
	&\ge \frac{2 \left(n^3+1-\floor*{c(n)}\right)}{n^2 (n^2 + 1)}.
	\end{align*}
Hence condition \eqref{eq:cond1} is satisfied. 

\smallskip
\textsf{Case 2:} $x \le \frac{(n+1)^2}{4}$ and $y > \frac{n^4-n^3+n^2}{2}$. 

Since $f(x,y)$ increases with $y$, we have
	\begin{align*}
	f(x,y) &\ge f\left(x, \frac{n^4-n^3+n^2}{2}\right)\\
	&= \frac{4x}{(n+1)^2}(n^3+1) + \left(1-\frac{4x}{(n+1)^2}\right)\frac{n^2(n^3+1)}{n^2+n-1} - x\\
	&= \left( \frac{4(n^3+1)}{(n+1)^2} - \frac{4n^2(n^3+1)}{(n+1)^2(n^2+n-1)} - 1\right)x + \frac{n^2(n^3+1)}{n^2+n-1}\\
	&= \left(\frac{4(n^3+1)(n-1)}{(n+1)^2(n^2+n-1)} - 1\right)x + \frac{n^2(n^3+1)}{n^2+n-1}.
	\end{align*}
Since $n \ge 3$, one can verify that the coefficient of $x$ in the last line is positive. Hence
$$
f(x,y) \ge \frac{n^2(n^3+1)}{n^2+n-1} \ge n^3+1-\floor*{c(n)}.
$$ 
This together with $x+y \le \frac{n^2 (n^2+1)}{2}$ yields
	\[
	\frac{f(x,y)}{x+y} \ge \frac{2 \left(n^3+1-\floor*{c(n)}\right)}{n^2 (n^2 + 1)}.
	\]
That is, condition \eqref{eq:cond1} is satisfied. 
	
\smallskip
\textsf{Case 3:} $\frac{(n+1)^2}{4} < x \le c(n)$. 

Then
	\[
	x - \left(1-\frac{(n+1)^2}{4x}\right)y' \le x \le c(n) < \floor*{c(n)}+1
	\]
and so condition \eqref{eq:cond2} is satisfied. 

\smallskip
\textsf{Case 4:} $c(n) < x \le n^2$. 

Since $g$ is concave and $c(n) \ge \frac{n(n+1)}{2}$, we have
	\begin{align*}
	\frac{g(x)-g(c(n))}{x-c(n)} &\ge \frac{g(n^2)-g(c(n))}{n^2-c(n)}\\
	&= \frac{1}{1-\frac{c(n)}{n^2}}\\
	&\ge \frac{1}{1 - \frac{(n+1)^2}{4c(n)}}\\
	&\ge \frac{1}{1-\frac{(n+1)^2}{4x}}.
	\end{align*}
On the other hand,
	\begin{align*}
	y' &\ge \frac{x(2n^2+1-x)}{2}-y\\
	&\ge \frac{x(2n^2+1-x)}{2}+x-\frac{n^4+n^2}{2}\\
	&= g(x)-g(c(n)).
	\end{align*}
So
	\[
	x -\left(1-\frac{(n+1)^2}{4x}\right)y' \le x-(x-c(n)) = c(n) < \floor*{c(n)}+1
	\]
and \eqref{eq:cond2} is satisfied. 
	
\smallskip
\textsf{Case 5:}  $x > n^2$. 

Let
	\[
	h(z) = z - \left(1-\frac{(n+1)^2}{4z}\right)\left(\frac{n^4 z}{n^2-1+z}+z-\frac{n^4+n^2}{2}\right).
	\]
	Then
	\[
	x - \left(1-\frac{(n+1)^2}{4x}\right)y' \le h(x).
	\]
Thus, to show that \eqref{eq:cond2} is satisfied, it suffices to prove $h(x)\le c(n)$. We achieve this by showing that $h(n^2) \le c(n)$ and that the function $h(z)$ is decreasing with $z$. 

To justify the first claim, we notice that $(3n^2-1)(3n+1) \ge 12(2n^2-1)$. Hence
	\begin{align*}
	h(n^2) &= \left(1-\frac{(n+1)^2}{4n^2}\right)\left(\frac{n^6}{2n^2-1}+n^2-\frac{n^4+n^2}{2}\right)\\
	&= n^2 - \frac{(3n^2-1)(3n+1)(n-1)}{8(2n^2-1)}\\
	&= n^2 - \frac{3n^2-1}{2n^2-1} \frac{3n+1}{8} (n-1)\\
	&\le n^2 - \frac{3}{2}(n-1)\\
	&\le c(n).
	\end{align*}
To prove the second claim, we take the derivative
	\begin{align*}
	\frac{dh}{dz} &= 1 - \frac{(n+1)^2}{4z^2}\left(\frac{n^4 z}{n^2-1+z}+z-\frac{n^4+n^2}{2}\right)\\
	&\quad\quad- \left(1-\frac{(n+1)^2}{4z}\right)\left(\frac{n^4(n^2-1)}{(n^2-1+z)^2}+1\right)\\
	&= \frac{n^2(n+1)}{8} \left(\frac{(n^2+1)(n+1)}{z^2} - \frac{2n^2(5n-3)}{(n^2-1+z)^2}\right).
	\end{align*}
It can be easily verified that $n^2(5n-3) \ge 2(n^2+1)(n+1)$. Thus
	\[
	\frac{2n^2(5n-3)}{(n^2+1)(n+1)} \ge 4 \ge \left(\frac{n^2-1}{z}+1\right)^2
	\]
and therefore $\frac{dh}{dz} \le 0$. So $h(z)$ is decreasing with $z$. 

In summary, in each case above either \eqref{eq:cond1} or \eqref{eq:cond2} holds for any $S \subseteq V$ with $1 \le |S| \le \frac{|V|}{2}$. Thus, by Lemma \ref{lem:two}, we obtain the lower bound in \eqref{eq:lub}. 

This completes the proof of Theorem \ref{thm:inc}.

\section{Proof of Theorem \ref{thm:noninc}}
\label{sec:noninc_unital}

In this section we assume that $U=(P,\mathcal{B})$ is a unital of order $n$ and $G$ its non-incidence graph $G_{\overline{U}}$. Denote the vertex set of $G$ by $V = P \cup \mathcal{B}$.  

We first prove that the right-hand side of \eqref{eq:noninc} is an upper bound for $i(G)$. Consider the case $n=2$. Pick any point vertex $p \in P$. There are exactly eight blocks incident with this point. Let $S \subseteq V$ consist of this point $p$, the eight blocks incident to it, and one other block. Then $|S|=10$ and $|N(S)| \le 8$ since $N(S) \subseteq P \setminus \{p\}$. Hence $i(G) \le \frac{|N(S)|}{|S|} \le \frac{4}{5}$. Now assume that $n \ge 3$. Let $S$ be any subset of $\mathcal{B}$ of size $\frac{n^4+n^2}{2}$. Then $|N(S)| \le n^3+1$ as $N(S) \subseteq P$. Therefore, $i(G) \le \frac{|N(S)|}{|S|} \le \frac{2(n^3+1)}{n^4+n^2}$.

Next we prove that the right-hand side of \eqref{eq:noninc} is also a lower bound for $i(G)$. Let $S$ be an arbitrary subset of $V$ with $1 \le |S| \le \frac{|V|}{2}$, and let $X = S \cap P$ and $Y = S \cap \mathcal{B}$. Let $x$, $y$, $x'$, and $y'$ be defined as in \eqref{eq:shorthand}. In view of \eqref{eq:uni}, the parameters of $\overline{U}$ are $(v,b,r,k,\lambda) = (n^3+1, n^4-n^3+n^2, n^4-n^3, n^3-n, n^4-n^3-n^2+1)$.

Consider the case $n=2$ first, for which the parameters of $\overline{U}$ are $(v,b,r,k,\lambda) = (9,12,8,6,5)$. In this case we have $x+y \le 10$, and the second and third inequalities of Corollary \ref{lem:keycorollary} reduce to $y+y' \ge \frac{64x}{3+5x}$ and $x+x' \ge \frac{48y}{4+5y}$ respectively. In order to prove $i(G) \ge \frac{4}{5}$, it suffices to show that $\frac{x'+y'}{x+y} \ge 1$ or $x' + y' \ge 8$. We split this into four cases:
\begin{enumerate}[(a)]
	\item If $x \le 1$ and $y \le 4$, then $y+y' \ge \frac{64x}{3+5x} \ge 2x$ and $x+x' \ge \frac{48y}{4+5y} \ge 2y$ so that
	$\frac{x'+y'}{x+y} = \frac{(x+x')+(y+y')}{x+y} - 1 \ge 1$.
	\item If $x \le 1$ and $y \ge 5$, then $x+x' \ge \frac{240}{29} > 8$. Since $x+x'$ is an integer, it follows that $x+x'\ge 9$ and thus $x' \ge 8$.
	\item If $x \ge 2$ and $y \le 1$, then $y+y' \ge \frac{128}{13} > 10$. This implies that $y+y' \ge 11$, from which it follows that $y' \ge 10$.
	\item Finally, if $x \ge 2$ and $y \ge 2$, then we still have $y+y' \ge 11$ from before. Additionally, we also have $x+x' \ge \frac{96}{14} > 6$, from which it follows that $x+x' \ge 7$ and thus $x'+y' = (x+x')+(y+y')-(x+y)\ge 8$.
\end{enumerate}

Now assume that $n \ge 3$. Since $\frac{2(n^3+1)}{n^2 (n^2 + 1)} < 1$ and $x+y \le \frac{n^2 (n^2 + 1)}{2}$, in order to prove $i(G) \ge \frac{2(n^3+1)}{n^2 (n^2 + 1)}$, it suffices to prove that $\frac{x'+y'}{x+y} \ge 1$ or $x'+y' \ge n^3+1$. This is proved as follows, again using the second and third inequalities of Corollary \ref{lem:keycorollary}.
\begin{enumerate}[(a)]
	\item If $y=0$, then $\frac{x'+y'}{x+y} \ge \frac{y+y'}{x} \ge \frac{r^2}{r+\lambda(x-1)} \ge \frac{r^2}{r+\lambda(v-1)} = \frac{r}{k} \ge 1$.
	
	\item If $x=0$ and $y \le n^2$, then $\frac{x'+y'}{x+y} \ge \frac{x+x'}{y} \ge \frac{rk}{r^2 - \lambda(b-y)} \ge \frac{rk}{r^2 - \lambda(b-n^2)} = n \ge 1$.
	
	\item If $x=0$ and $y > n^2$, then $x+x' \ge \frac{rky}{r^2-\lambda(b-y)} \ge \left(\frac{rk}{r^2-\lambda(b-n^2)}\right)y = ny > n^3$. Since $x'$ is an integer, it follows that $x' = n^3+1$.
	
	\item If $x \ge 1$ and $y \ge 1$, then $y+y' \ge r$ and $x+x' \ge k$, so that
	\[
	x'+y' = (x+x')+(y+y')-(x+y) \ge n^4-n - \frac{n^4+n^2}{2} \ge n^3+1. \qedhere
	\]
\end{enumerate}

This completes the proof of Theorem \ref{thm:noninc}.

\bigskip
\noindent \textbf{Acknowledgement}~~A. Hui was supported by the Young Scientists Fund (Grant No. 11701035) of the National Natural Science Foundation of China. S. Zhou was supported by a Future Fellowship (FT110100629) of the Australian Research Council.

\end{document}